\documentclass[12pt]{amsart}

\usepackage{mathrsfs,amssymb}

\newtheorem{theorem}{Theorem}[section]
\newtheorem{lemma}[theorem]{Lemma}

\newtheorem{cor}[theorem]{Corollary}

\theoremstyle{definition}

\theoremstyle{remark}
\newtheorem{remark}[theorem]{Remark}

\numberwithin{equation}{section}

\let \la=\lambda
\let \e=\varepsilon
\let \d=\delta
\let \o=\omega
\let \a=\alpha
\let \f=\varphi

\let \O=\Omega

\let \ga=\gamma

\begin{document}
\title[A weak type estimate]
{A weak type estimate for rough singular integrals}

\author{Andrei K. Lerner}
\address{Department of Mathematics,
Bar-Ilan University, 5290002 Ramat Gan, Israel}
\email{lernera@math.biu.ac.il}

\thanks{The  author was supported by the Israel Science Foundation (grant No. 447/16).}

\begin{abstract}
We obtain a weak type $(1,1)$ estimate for a maximal operator associated with the classical rough homogeneous singular integrals $T_{\O}$.
In particular, this provides a different approach to a sparse domination for $T_{\O}$ obtained recently by Conde-Alonso, Culiuc, Di Plinio and Ou \cite{CCPO}.
\end{abstract}


\maketitle

\section{Introduction}
In this paper we consider a class of rough homogeneous singular integrals defined by
$$T_{\Omega}f(x)=\text{p.v.}\int_{{\mathbb R}^n}f(x-y)\frac{\O(y/|y|)}{|y|^n}dy,$$
with $\O\in L^{q}(S^{n-1}), q\ge 1,$ having zero average.

Calder\'on and Zygmund \cite{CZ} proved that if $\O\in L\log L(S^{n-1})$, then $T_{\O}$ is bounded on $L^p$ for all $1<p<\infty$. The weak type $(1,1)$ of $T_{\O}$
was established by Christ \cite{C} and Hofmann \cite{H} in the case $n=2$ and $\O\in L^q(S^1), q>1$, and by Christ and Rubio de Francia \cite{CHR} for $\O\in L\log L(S^1)$. Finally Seeger
\cite{S} proved that $T_{\O}$ is weak $(1,1)$ bounded for $\O\in L\log L(S^{n-1})$ in all dimensions.

Notice that contrary to singular integrals with smooth kernels, for rough singular integrals the question whether the maximal singular integral operator
$$T_{\O}^{\star}f(x)=\sup_{\e>0}|T_{\O}(f\chi_{\{|\cdot|>\e\}})(x)|$$
is of weak type $(1,1)$ is still open even in the case when $\O\in L^{\infty}(S^{n-1})$.

On the other hand, weak type estimates of maximal truncations are important for the so-called sparse domination.
Sparse bounds for different operators is a recent trend in Harmonic Analysis (see, e.g., \cite{BFP,CCPO,CPO,PDU,La,La1,LS,L2},
and this list is far from complete). By sparse bounds one typically means a domination of the bilinear form $|\langle Tf,g \rangle|$ (for a given operator $T$)
by
$$\Lambda(f,g)=\sum_{Q\in {\mathcal S}}\langle f\rangle_{r,Q}\langle g\rangle_{s,Q}|Q|,$$
with suitable $1\le r,s<\infty$, where $\langle f\rangle_{p,Q}=\left(\frac{1}{|Q|}\int_Q|f|^p\right)^{1/p},$
and ${\mathcal S}$ is a sparse family of cubes from ${\mathbb R}^n$. We say that ${\mathcal S}$ is $\eta$-sparse, $0<~\eta~\le~1$, if for every cube
$Q\in {\mathcal S}$, there exists a measurable set $E_Q\subset Q$ such that $|E_Q|\ge \eta|Q|$, and the sets
$\{E_Q\}_{Q\in {\mathcal S}}$ are pairwise disjoint.
The advantage of sparse bounds is that they easily imply quantitative weighted estimates in terms of the Muckenhoupt and reverse H\"older constants.

In \cite{L2}, the following principle was established: if $T$ is a sublinear operator of weak type $(p,p)$, and the maximal operator
$$M_Tf(x)=\sup_{Q\ni x}\|T(f\chi_{{\mathbb R}^n\setminus 3Q})\|_{L^{\infty}(Q)}$$
is of weak type $(r,r)$, for some $1\le p\le r<\infty$, then $T$ is dominated pointwise by the sparse operator $\sum_{Q\in {\mathcal S}}\langle f\rangle_{r,Q}\chi_Q$.

While this principle perfectly works for smooth singular integrals, it seems to be not as useful for rough singular integrals $T_{\O}$.
Indeed, in this case the lack of smoothness does not allow to handle the $L^{\infty}$ norm appearing in the definition of $M_{T_{\Omega}}$
in an efficient way. Observe also that for $\O\in L^{\infty}(S^{n-1})$ the operator $M_{T_{\O}}$ may be as large as $T_{\O}^{\star}$, and therefore the weak type $(1,1)$ for $M_{T_{\O}}$ is a difficult open question.

Recently, Conde-Alonso, Culiuc, Di Plinio and Ou \cite{CCPO} obtained another sparse domination principle, not relying on the end-point
weak type estimates of maximal truncations. This principle was effectively applied to rough singular integrals. For example, if $\O\in L^{\infty}$, the following estimate in \cite{CCPO} was proved for all $1<p<\infty$:
\begin{equation}\label{rough}
|\langle T_{\O}f,g \rangle|\le C_np'\|\O\|_{L^{\infty}(S^{n-1})}\sup_{\mathcal S}\sum_{Q\in {\mathcal S}}\langle f\rangle_{p,Q}\langle g\rangle_{1,Q}|Q|,
\end{equation}
where $p'=\frac{p}{p-1}$. This estimate recovers the quantitative weighted bound
\begin{equation}\label{quant}
\|T_{\O}\|_{L^2(w)\to L^2(w)}\le C_n\|\O\|_{L^{\infty}(S^{n-1})}[w]_{A_2}^2,
\end{equation}
obtained earlier by Hyt\"onen, Roncal and Tapiola \cite{HRT}.

The dependencies on $[w]_{A_2}$ in (\ref{quant}) and on $p$ in (\ref{rough}) when $p\to 1$ are closely related.
At this time, we do not know whether the quadratic dependence on $[w]_{A_2}$ in (\ref{quant}) can be improved. By this reason, it is also
unknown whether the dependence on $p$  in (\ref{rough}) is sharp.

In this paper we present a different approach to (\ref{rough}) and (\ref{quant}) based on weak type estimates of suitable maximal operators.
We believe that this approach is of independent interest in the theory of the rough singular integrals.

Given an operator $T$, define the maximal operator $M_{\la, T}$ by
$$M_{\la, T}f(x)=\sup_{Q\ni x}(T(f\chi_{{\mathbb R}^n\setminus 3Q})\chi_Q)^*(\la|Q|)\quad(0<\la<1),$$
where the supremum is taken over all cubes $Q\subset {\mathbb R}^n$ containing the point $x$, and $f^*$ denotes the non-increasing rearrangement of $f$.

Assume that $T$ is of weak type $(1,1)$. Then it is easy to show (just using that $T(f\chi_{{\mathbb R}^n\setminus 3Q})=Tf-T(f\chi_{3Q})$ along with the standard estimates of
the maximal operators) that $M_{\la,T}$ is of weak type $(1,1)$ too, and
$$
\|M_{\la,T}\|_{L^1\to L^{1,\infty}}\le \frac{C_n}{\la}\|T\|_{L^1\to L^{1,\infty}}\quad(0<\la<1).
$$
On the other hand, $M_{\la,T}\uparrow M_T$ as $\la\to 0$. Therefore, the weak type $(1,1)$ of $M_T$ (which, by \cite{L2}, leads to the
best possible sparse domination of $T$) is equivalent to the weak type $(1,1)$ of $M_{\la,T}$ with the
$\|M_{\la,T}\|_{L^1\to L^{1,\infty}}$ norm bounded in $\la$. These observations raise a natural question about
the sharp dependence of $\|M_{\la,T}\|_{L^1\to L^{1,\infty}}$ on $\la$ for a given operator $T$ of weak type $(1,1)$.
More generally, if $T$ is of weak type $(p,p)$, one can ask about the sharp dependence of $\|M_{\la,T}\|_{L^p\to L^{p,\infty}}$ on $\la$.

The main result of this paper is the following estimate for rough homogeneous singular integrals $T_{\O}$.

\begin{theorem}\label{logest} If $\O\in L^{\infty}(S^{n-1})$, then
\begin{equation}\label{logestim}
\|M_{\la, T_{\O}}\|_{L^1\to L^{1,\infty}}\le C_n\|\O\|_{L^{\infty}(S^{n-1})}\Big(1+\log\frac{1}{\la}\Big)\quad(0<\la<1).
\end{equation}
\end{theorem}

The proof of Theorem \ref{logest} is given in the next Section. In Section~3, we obtain a sparse domination principle, where the operator
$M_{\la, T}$ plays an important role. In particular, we will show that (\ref{logestim}) implies the sparse bound (\ref{rough}).

Notice that any improvement of the logarithmic dependence in
(\ref{logestim}) would lead to the corresponding improvement of the dependence on $p$ when $p\to 1$ in (\ref{rough}). Therefore, by the reasons discussed above,
we do not know whether the logarithmic dependence in (\ref{logestim}) can be improved.

\section{Proof of Theorem \ref{logest}}
\subsection{An overview of the proof} As we have mentioned before, for smooth singular integrals one can use the trivial
estimate $M_{\la,T}\le M_T$, which yields the $L^1\to L^{1,\infty}$ bound with no dependence on $\la$. This simple idea
suggests to approximate a rough singular integral $T_{\O}$ by smooth ones. Given $0<\e<1$, we decompose
\begin{equation}\label{decomp}
T_{\O}=T_{\O_{\e}}+T_{\O-\O_{\e}},
\end{equation}
where $T_{\O_{\e}}$ is a smooth singular integral to which the standard Calder\'on-Zygmund theory is applicable,
and $\|T_{\O-\O_{\e}}\|_{L^2\to L^2}$ satisfies a good estimate in terms of $\e$ when $\e\to 0$.
Then
$$M_{\la,T_{\O}}\le M_{T_{\O_{\e}}}+M_{\la,T_{\O-\O_{\e}}}.$$

For the smooth part $T_{\O_{\e}}$ we use a very similar analysis to what was done by Hyt\"onen, Roncal and Tapiola \cite{HRT}, namely, we show that the kernel of
$T_{\O_{\e}}$ is Dini-continuous tracking the Dini constant, which implies
$$\|M_{T_{\O_{\e}}}\|_{L^1\to L^{1,\infty}}\le C_n\|\Omega\|_{L^{\infty}(S^{n-1})}\log\frac{2}{\e}\quad(0<\e<1).$$

The non-smooth part $T_{\O-\O_{\e}}$ is more complicated. The only fact that $T_{\O-\O_{\e}}$ is a singular integral with small $L^2$ norm in terms of $\e$ is not enough
in order to obtain a good estimate for $\|M_{\la,T_{\O-\O_{\e}}}\|_{L^1\to L^{1,\infty}}$ in terms of $\e$ and $\la$.
However, we can keep $\O_{\e}$ in (\ref{decomp}) to be homogeneous. This allows to apply to $T_{\O-\O_{\e}}$ the deep machinery developed by Seeger in \cite{S}.
Combining it with several other ingredients, we obtain
$$
\|M_{\la,T_{\O-\O_{\e}}}\|_{L^1\to L^{1,\infty}}\le C_n\|\Omega\|_{L^{\infty}(S^{n-1})}\left(\frac{\e^{1/2}}{\la}+\log\frac{2}{\e}\right).
$$
It remains to optimize the obtained estimates with respect to $\e$, namely, we take $\e=\la^{2}$.

The details follow in next subsections.

\subsection{Main splitting}
Denote $$\O_0(x)=\frac{\Omega(x/|x|)}{|x|^n}\chi_{\{1\le|x|\le 2\}}(x).$$
Let $\f\in C^{\infty}$, $\text{supp}\,\f\subset\{|x|<1\}$ and $\int\f=1$. For $\e>0$, set
$$\O_{\e}(\theta)=\frac{1}{\log 2}\int_0^{\infty}\O_0*\f_{\e}(t\theta)t^{n-1}dt\quad(\theta\in S^{n-1}),$$
where $\f_{\e}(x)=\frac{1}{\e^n}\f(x/\e)$.

We split $T_{\O}$ as follows: $T_{\O}=T_{\O_{\e}}+T_{\O-\O_{\e}}$.

\begin{lemma}\label{L2} For every $0<\a<1$,
$$\|T_{\O-\O_{\e}}\|_{L^2\to L^2}\le C_{\a,n}\|\O\|_{L^{\infty}(S^{n-1})}\e^{\a}\quad(0<\e<1).$$
\end{lemma}

\begin{proof} Observe that the kernel of $T_{\O-\O_{\e}}$ is given by
$$\Psi_{\e}(x)=\frac{(\O-\O_{\e})(x/|x|)}{|x|^n}=\frac{1}{\log 2}\int_0^{\infty}(\O_0-\O_0*\f_{\e})(tx)t^{n-1}dt.$$
Hence, by Plancherel's theorem, it suffices to show that
\begin{equation}\label{sufs}
\|\widehat \Psi_{\e}\|_{L^{\infty}}\le C_{\a,n}\|\O\|_{L^{\infty}(S^{n-1})}\e^{\a},
\end{equation}
where the Fourier transform is taken in the appropriate principal value sense.

We will use the following well known estimate (see \cite{DR}):
$$|\widehat \O_0(\xi)|\le C_{\a,n}\|\O\|_{L^{\infty}(S^{n-1})}\min(|\xi|,|\xi|^{-\a})\quad(0<\a<1).$$
Also, since $\int\f=1$, we have
$$|\widehat\f(\xi)-1|\le C\min(|\xi|,1),$$
with some absolute $C>0$. Combining these estimates yields
\begin{eqnarray*}
&&|\widehat \Psi_{\e}(\xi)|\le \frac{1}{\log 2}\int_0^{\infty}|\widehat \O_0(\xi/t)||\widehat \f(\e\xi/t)-1|\frac{dt}{t}\\
&&\le C\|\Omega\|_{L^{\infty}}\int_0^{\infty}\min(|\xi/t|,|\xi/t|^{-\a})\min(|\e\xi/t|,1)\frac{dt}{t}\\
&&=C\|\Omega\|_{L^{\infty}}\int_0^{\infty}\min(1/t,t^{\a})\min(\e/t,1)\frac{dt}{t}\\
&&\le C\|\Omega\|_{L^{\infty}}\Big(\int_0^1\min(\e/t,1)\frac{dt}{t^{1-\a}}+\e\Big)\le
C\|\Omega\|_{L^{\infty}}\e^{\a},
\end{eqnarray*}
which proves (\ref{sufs}).
\end{proof}

\subsection{Calder\'on-Zygmund theory of $T_{\O_{\e}}$}
Let $Tf=\text{p.v.}f*K$ be $L^2$ bounded with $K$ satisfying $|K(x)|\le \frac{C_K}{|x|^n}$ and
\begin{equation}\label{smcond}
|K(x-y)-K(x)|\le \o(|y|/|x|)\frac{1}{|x|^n}\quad(|y|<|x|/2),
\end{equation}
where
$$[\o]_{\text{Dini}}=\int_0^1\o(t)\frac{dt}{t}<\infty.$$

It was proved in \cite[Lemma 3.2]{L2} that
$$M_Tf(x)\le C_n([\omega]_{\text{\rm{Dini}}}+C_K)Mf(x)+T^{\star}f(x),$$
where $M$ is the Hardy-Littlewood maximal operator, and $T^{\star}$ is the maximal singular integral.
The classical proof (see, e.g., \cite[Ch. 4.3]{G}) shows that
$$
\|T^{\star}\|_{L^1\to L^{1,\infty}}\le C_n(\|T\|_{L^2\to L^2}+C_K+[\omega]_{\text{\rm{Dini}}}).
$$
This, along with the previous estimate, implies
\begin{equation}\label{mt}
\|M_T\|_{L^1\to L^{1,\infty}}\le C_n(\|T\|_{L^2\to L^2}+C_K+[\omega]_{\text{\rm{Dini}}}).
\end{equation}

\begin{lemma}\label{czprop} The operator $T_{\O_{\e}}$ satisfies
$$\|M_{T_{\O_{\e}}}\|_{L^1\to L^{1,\infty}}\le C_n\|\Omega\|_{L^{\infty}(S^{n-1})}\log\frac{2}{\e}\quad(0<\e<1).$$
\end{lemma}

\begin{proof} Observe that for $0<\e<1$, $\text{supp}\,\O_0*\f_{\e}\subset\{|x|\le 3\}$. Also,
$$\|\O_0*\f_{\e}\|_{L^{\infty}}\le \|\O_0\|_{L^{\infty}}\|\f_{\e}\|_{L^1}\le C\|\O\|_{L^{\infty}}.$$
Therefore, setting $K_{\e}(x)=\frac{\O_{\e}(x/|x|)}{|x|^n}$, we obtain
\begin{equation}\label{linf}
C_{K_{\e}}=\|\O_{\e}\|_{L^{\infty}}\le \frac{1}{\log 2}\|\O_0*\f_{\e}\|_{L^{\infty}}\int_0^3t^{n-1}dt\le C_n\|\O\|_{L^{\infty}}.
\end{equation}
This, along with the standard $L^2$ bound (see \cite{CZ}), implies
\begin{equation}\label{L2bound}
\|T_{\O_{\e}}\|_{L^2\to L^2}\le  C_n\|\O_{\e}\|_{L^{\infty}}\le C_n\|\O\|_{L^{\infty}}.
\end{equation}

Further, using that
$$K_{\e}(x)=\frac{1}{\log 2}\int_0^{\infty}\O_0*\f_{\e}(tx)t^{n-1}dt,$$
we obtain
\begin{eqnarray*}
|\nabla K_{\e}(x)|&=&\frac{1}{\e}\frac{1}{\log 2}\int_0^{\infty}\O_0*(\nabla\f)_{\e}(tx)t^{n}dt\\
&=&\frac{1}{\e}\frac{1}{|x|^{n+1}}\frac{1}{\log 2}\int_0^{\infty}\O_0*(\nabla\f)_{\e}(tx/|x|)t^{n}dt.
\end{eqnarray*}
From this and from the same argument as used in the proof of (\ref{linf}),
$$|\nabla K_{\e}(x)|\le \frac{C_n}{\e}\|\O\|_{L^{\infty}}\frac{1}{|x|^{n+1}}.$$

Therefore, by the mean value theorem,
$$|K_{\e}(x-y)-K_{\e}(x)|\le \frac{C_n}{\e}\|\O\|_{L^{\infty}}\frac{|y|}{|x|}\frac{1}{|x|^n}\quad(|y|<|x|/2).$$
Also, by (\ref{linf}),
$$|K_{\e}(x-y)-K_{\e}(x)|\le C_n\|\O\|_{L^{\infty}}\frac{1}{|x|^n}\quad(|y|<|x|/2).$$
Hence, $K_{\e}$ satisfies (\ref{smcond}) with
$$\o(t)=C_n\|\O\|_{L^{\infty}}\min(1,t/\e),$$
which implies
$$[\o]_{\text{\rm{Dini}}}\le C_n\|\O\|_{L^{\infty}}\log\frac{2}{\e}.$$
This, along with (\ref{linf}), (\ref{L2bound}) and (\ref{mt}), completes the proof.
\end{proof}

\subsection{The key estimate} In order to handle the rough part $T_{\O-\O_{\e}}$, we will prove the following lemma which can be stated for
a general rough homogeneous singular integral $T_{\O}$ with $\O\in L^{\infty}(S^{n-1})$.

\begin{lemma}\label{keying}  There exists $C_n>0$ such that
for every $0<\d\le 1$,
$$\|M_{\la,T_{\O}}\|_{L^1\to L^{1,\infty}}\le C_n\Big(\frac{\d}{\la}+\log\frac{2}{\d}\Big)\max\Big(\|\O\|_{L^{\infty}(S^{n-1})},\frac{\|T_{\O}\|_{L^2\to L^2}}{\d}\Big).$$
\end{lemma}

Before proving Lemma \ref{keying}, let us show how to complete the proof of Theorem \ref{logest}.

By (\ref{linf}),
$$\|\O-\O_{\e}\|_{L^{\infty}}\le C_n\|\O\|_{L^{\infty}}.$$
This, combined with Lemmata \ref{L2} (where we take $\a=1/2$) and \ref{keying}, implies
$$\|M_{\la,T_{\O-\O_{\e}}}\|_{L^1\to L^{1,\infty}}\le C_n\|\O\|_{L^{\infty}}
\Big(\frac{\d}{\la}+\log\frac{2}{\d}\Big)\max(1,\e^{1/2}/\d).$$
Taking here $\d=\e^{1/2}$, we obtain
\begin{equation}\label{rp}
\|M_{\la,T_{\O-\O_{\e}}}\|_{L^1\to L^{1,\infty}}\le C_n\|\O\|_{L^{\infty}}
\Big(\frac{\e^{1/2}}{\la}+\log\frac{2}{\e}\Big).
\end{equation}

Since
$$M_{\la, T_{\O}}f(x)\le M_{T_{\O_{\e}}}(x)+M_{\la,T_{\O-\O_{\e}}}f(x),$$
by Lemma \ref{czprop} cobmined with (\ref{rp}),
\begin{eqnarray*}
\|M_{\la,T_{\O}}\|_{L^1\to L^{1,\infty}}&\le& 2(\|M_{T_{\O_{\e}}}\|_{L^1\to L^{1,\infty}}+\|M_{\la,T_{\O-\O_{\e}}}\|_{L^1\to L^{1,\infty}})\\
&\le& C_n\|\O\|_{L^{\infty}}
\Big(\frac{\e^{1/2}}{\la}+\log\frac{2}{\e}\Big).
\end{eqnarray*}
Finally, we take here $\e=\la^{2}$, and this completes the proof of Theorem~\ref{logest}.

We turn now to the proof of Lemma \ref{keying}.

\subsection{A reduction to dyadic case} It will be convenient to work with a dyadic version of $M_{\la, T_{\O}}$. We first state several preliminary facts
about dyadic lattices.

Given a cube $Q_0\subset {\mathbb R}^n$, let ${\mathcal D}(Q_0)$ denote the set of all dyadic cubes with respect to $Q_0$, that is, the cubes
obtained by repeated subdivision of $Q_0$ and each of its descendants into $2^n$ congruent subcubes.

A dyadic lattice ${\mathscr D}$ in ${\mathbb R}^n$ is any collection of cubes such that
\begin{enumerate}
\renewcommand{\labelenumi}{(\roman{enumi})}
\item
if $Q\in{\mathscr D}$, then each child of $Q$ is in ${\mathscr D}$ as well;
\item
every 2 cubes $Q',Q''\in {\mathscr D}$ have a common ancestor, i.e., there exists $Q\in{\mathscr D}$ such that $Q',Q''\in {\mathcal D}(Q)$;
\item
for every compact set $K\subset {\mathbb R}^n$, there exists a cube $Q\in {\mathscr D}$ containing $K$.
\end{enumerate}

For this definition, as well as for the next Theorem, we refer to \cite{LN}.

\begin{theorem}\label{three}{\rm{(The Three Lattice Theorem)}}
For every dyadic lattice ${\mathscr D}$, there exist $3^n$ dyadic lattices ${\mathscr D}^{(1)},\dots,{\mathscr D}^{(3^n)}$ such that
$$\{3Q: Q\in{\mathscr D}\}=\cup_{j=1}^{3^n}{\mathscr D}^{(j)}$$
and for every cube $Q\in {\mathscr D}$ and $j=1,\dots,3^n$, there exists a unique cube $R\in {\mathscr D}^{(j)}$ of
sidelength $\ell_{R}=3\ell_Q$ containing $Q$.
\end{theorem}

Turn now to the definition of $M_{\la,T_{\O}}$.
Fix a dyadic lattice ${\mathscr D}$. Let $Q$ be an arbitrary cube containing the point $x$. There exists a cube $R\in {\mathscr D}$
containing the center of $Q$ and such that $\ell_Q/2<\ell_R\le \ell_Q$ (by $\ell_Q$ we denote the sidelength of $Q$).
Then $Q\subset 3R$, and hence $3Q\subset 9R$.
For every $\xi\in Q$,
$$|T_{\O}(f\chi_{9R\setminus 3Q})(\xi)|\le C_n\|\O\|_{L^{\infty}}\frac{1}{|9R|}\int_{9R}|f|\le C_n\|\O\|_{L^{\infty}}Mf(x).$$
Hence,
\begin{eqnarray}
&&\big(T_{\Omega}(f\chi_{{\mathbb R}^n\setminus 3Q})\chi_Q\big)^*\big(\la|Q|\big)\label{a1}\\
&&\le \big(T_{\Omega}(f\chi_{{\mathbb R}^n\setminus 9R})\chi_{3R}\big)^*\big(\la|R|\big)+C_n\|\O\|_{L^{\infty}}Mf(x).\nonumber
\end{eqnarray}

By Theorem \ref{three}, there exists a dyadic
lattice ${\mathscr D}^{(j)}, j=1,\dots 3^n$ such that $3R\in {\mathscr D}^{(j)}$.
Applying Theorem \ref{three} again, we obtain that there are dyadic lattices ${\mathscr D}^{(j,i)}$ such that
$$\{3Q: Q\in{{\mathscr D}^{(j)}}\}=\cup_{i=1}^{3^n}{\mathscr D}^{(j,i)}$$
Hence, setting
$${{\mathscr E}^{(j,i)}}=\{Q\in {\mathscr D}^{(j)}:3Q\in {\mathscr D}^{(j,i)}\},$$
by (\ref{a1}), we obtain
\begin{equation}\label{mlat}
M_{\la, T_{\O}}f(x)\le \sum_{i,j=1}^{3^n}M_{\la/3^n, T_{\Omega}}^{{\mathscr E}^{(j,i)}}f(x)+ C_n\|\O\|_{L^{\infty}}Mf(x),
\end{equation}
where
$$M^{{\mathscr E}^{(j,i)}}_{\la, T_{\Omega}}f(x)=\sup_{Q\ni x:Q\in {\mathscr D}^{(j)},3Q\in {\mathscr D}^{(j,i)}}\big(T_{\Omega}(f\chi_{{\mathbb R}^n\setminus 3Q})\chi_Q\big)^*\big(\la|Q|\big).$$

Fix now two dyadic lattices ${\mathscr D}$ and ${\mathscr D}'$. Let ${\mathcal F}$ be any finite family of cubes $Q$ from ${\mathscr D}$ such that $3Q\in {\mathscr D}'$.
By (\ref{mlat}), by the weak type $(1,1)$ of $M$, and by the monotone convergence theorem, it suffices to prove Lemma \ref{keying} for the dyadic
version of $M_{\la,T_{\O}}$ defined by
$$M^{\mathcal F}_{\la, T_{\O}}f(x)=\begin{cases}
\displaystyle\max_{Q\ni x, Q\in {\mathcal F}}\big(T_{\Omega}(f\chi_{{\mathbb R}^n\setminus 3Q})\chi_Q\big)^*\big(\la|Q|\big), & x\in \cup_{Q\in {\mathcal F}}Q\\
0, &\text{otherwise}.\end{cases}
$$

\subsection{The Calder\'on-Zygmund splitting}
Let $f\in L^1({\mathbb R}^n)$ and let $\a>0$. Apply the Calder\'on-Zygmund decomposition to $f$ at height $A\a$ formed by the cubes from ${\mathscr D}'$,
where $A>0$ will be specified later.
To be more precise, let $M^{{\mathscr D}'}$ be the dyadic maximal operator with respect to ${\mathscr D}'$. Let ${\mathcal P}$ be a family of the maximal
pairwise disjoint cubes forming the set $\{x:M^{{\mathscr D}'}f(x)>A\a\}$. For a cube $P\in {\mathcal P}$ set $b_P=(f-\frac{1}{|P|}\int_Pf)\chi_P$. Next, let
$b=\sum_{P\in {\mathcal P}}b_P$ and $g=f-b$. We have
\begin{equation}\label{embed}
|\{M^{\mathcal F}_{\la, T_{\Omega}}f>\a\}|\le |\{M^{\mathcal F}_{\la/2, T_{\Omega}}g>\a/2\}|+
|\{M^{\mathcal F}_{\la/2, T_{\Omega}}b>\a/2\}|
\end{equation}
(notice that here we have used the standard property of the rearrangement saying that $(f+g)^*(t)\le f^*(t/2)+g^*(t/2)$).

For the good part, we will use the following simple lemma.

\begin{lemma}\label{l2} Assume that $T$ is a sublinear, $L^2$ bounded operator. Then
$$\|M_{\la, T}f\|_{L^{2,\infty}}\le \frac{C_n}{\la^{1/2}}\|T\|_{L^2\to L^2}\|f\|_{L^2}\quad(0<\la<1).$$
\end{lemma}

\begin{proof} Let $x\in Q$. Then, by Chebyshev's inequality,
\begin{eqnarray*}
&&\big(T(f\chi_{{\mathbb R}^n\setminus 3Q})\chi_Q\big)^*\big(\la|Q|\big)\\
&&\le \big((Tf)\chi_Q\big)^*\big(\la|Q|/2\big)+\big(T(f\chi_{3Q})\big)^*\big(\la|Q|/2\big)\\
&&\le\frac{C_n}{\la^{1/2}}\big(M_2(Tf)(x)+\|T\|_{L^2\to L^2}M_2f(x)\big),
\end{eqnarray*}
where $M_2f(x)=M(|f|^2)(x)^{1/2}$. Therefore,
$$M_{\la, T}f(x)\le \frac{C_n}{\la^{1/2}}\big(M_2(Tf)(x)+\|T\|_{L^2\to L^2}M_2f(x)\big),$$
which, along with the $L^2\to L^{2,\infty}$ boundedness of $M_2$, completes the proof.
\end{proof}

Since $\|g\|_{L^{\infty}}\le 2^nA\a$ and $\|g\|_{L^1}\le \|f\|_{L^1}$, by Lemma \ref{l2},
\begin{eqnarray}
|\{x:M^{\mathcal F}_{\la/2,T_{\O}}g(x)>\a/2\}|&\le& \frac{C_n}{\la\a^2}\|T_{\Omega}\|_{L^2\to L^2}^2\int_{{\mathbb R}^n}|g|^2dx\label{goodpart}\\
&\le& \frac{C_nA\|T_{\Omega}\|_{L^2\to L^2}^2}{\la}\frac{\|f\|_{L^1}}{\a}.\nonumber
\end{eqnarray}

\subsection{Estimate of the bad part}
Pick $\psi\in C^{\infty}({\mathbb R})$ such that $\text{supp}\,\psi\subset [1/2,2]$ and $\sum_{j\in {\mathbb Z}}\psi(2^{-j}t)\equiv 1$ for all $t\not=0$.
Denote $K(x)=\frac{\O(x/|x|)}{|x|^n}$, and set
$K_j(x)=\psi(2^{-j}|x|)K(x)$ and $B_l=\sum_{|P|=2^{nl}}b_P(x)$. Then $K=\sum_{j\in {\mathbb Z}}K_j$ and $b=\sum_{l\in {\mathbb Z}}B_l$.

Assume that $|P|=2^{n(j-s)}$ and $x\not\in \ga P, \ga>1$. Since
$$\text{dist}\,((\ga P)^c,P)=\frac{\ga-1}{2}\ell_P=\frac{\ga-1}{2}2^{j-s},$$
we obtain that if $2^s<\frac{\ga-1}{4}$, then $\text{dist}\,(x,P)>2^{j+1}$, and therefore $|K_j|*|b_P|(x)=0$.
Setting in this argument $\ga=9$, we conclude that for every cube $Q$,
\begin{equation}\label{emb}
\text{supp}\Big(\sum_{s<1}\sum_{j\in {\mathbb Z}}K_j*(B_{j-s}\chi_{{\mathbb R}^n\setminus 3Q})\Big)\subset \bigcup_{P\in {\mathcal P}}9P.
\end{equation}

Set now $E=\cup_{P\in {\mathcal P}}9P$ and $E^*=\{x:M^{\mathscr D}\chi_{E}(x)>\la/8\}$.
Observe that
\begin{eqnarray}
|E^*|\le \frac{8}{\la}|E|&\le& \frac{9^{n+1}}{\la}|\cup_{P\in {\mathcal P}}P|\label{estem}\\
&\le& \frac{9^{n+1}}{\la A}\frac{\|f\|_{L^1}}{\a}.\nonumber
\end{eqnarray}

Assume that $x\not \in E^*$, and let $Q\in {\mathcal F}$, $x\in Q$. Then $|Q\cap E|\le \frac{\la}{8}|Q|$, and hence, by (\ref{emb}),
$$
\Big(\sum_{s<1}\sum_{j\in {\mathbb Z}}K_j*(B_{j-s}\chi_{{\mathbb R}^n\setminus 3Q})\chi_Q\Big)^*(\la|Q|/4)=0.
$$
Therefore, using that
$$
T_{\Omega}(b\chi_{{\mathbb R}^n\setminus 3Q})(x)=\sum_{s\in {\mathbb Z}}\sum_{j\in {\mathbb Z}}K_j*(B_{j-s}\chi_{{\mathbb R}^n\setminus 3Q})(x),
$$
we obtain
\begin{eqnarray}
&&(T_{\O}(b\chi_{{\mathbb R}^n\setminus 3Q})\chi_Q)^*(\la|Q|/2)\label{ineq}\\
&&\le \Big(\sum_{s\ge 1}\sum_{j\in {\mathbb Z}}K_j*(B_{j-s}\chi_{{\mathbb R}^n\setminus 3Q})\chi_Q\Big)^*(\la|Q|/4)\nonumber.
\end{eqnarray}

Let $m\in {\mathbb N}, m\ge 2$, that will be specified later. Set
$${\mathcal M}_{\la}b(x)=\max_{Q\ni x, Q\in {\mathcal F}}\Big(\sum_{s\ge m}\sum_{j\in {\mathbb Z}}K_j*(B_{j-s}\chi_{{\mathbb R}^n\setminus 3Q})\chi_Q\Big)^*(\la|Q|/4)$$
for $x\in \cup_{Q\in {\mathcal F}}Q$, and ${\mathcal M}_{\la}b(x)=0$ otherwise.
Set also $T_jf(x)=K_j*f$. Then
\begin{eqnarray*}
&&\Big(\sum_{s\ge 1}\sum_{j\in {\mathbb Z}}K_j*(B_{j-s}\chi_{{\mathbb R}^n\setminus 3Q})\chi_Q\Big)^*(\la|Q|/4)\\
&&\le \sum_{s=1}^{m-1}\sum_{j\in {\mathbb Z}}M_{T_j}(B_{j-s})(x)+{\mathcal M}_{\la}b(x),
\end{eqnarray*}
which, along with (\ref{ineq}), yields
$$M^{\mathcal F}_{\la/2, T_{\Omega}}b(x)\le \sum_{s=1}^{m-1}\sum_{j\in {\mathbb Z}}M_{T_j}(B_{j-s})(x)+{\mathcal M}_{\la}b(x)\quad(x\not\in E^*).$$
Therefore,
\begin{eqnarray}
&&|\{x:M^{\mathcal F}_{\la/2, T_{\Omega}}b(x)>\a/2\}|\le |E^*|\label{estol}\\
&&+\frac{4}{\a}\sum_{s=1}^{m-1}\sum_{j\in {\mathbb Z}}\|M_{T_j}(B_{j-s})\|_{L^1}+|\{x:{\mathcal M}_{\la}b(x)>\a/4\}|.\nonumber
\end{eqnarray}

\subsection{Estimate of $M_{T_j}$}
\begin{lemma}\label{mtj} The operator $M_{T_j}$ is $L^1$ bounded, and
\begin{equation}\label{estmt}
\|M_{T_j}(f)\|_{L^1}\le C_n\|\Omega\|_{L^{\infty}}\|f\|_{L^1}.
\end{equation}
\end{lemma}

\begin{proof}
Let $x,\xi\in Q$. If $\ell_Q>2^{j-1}$, then $\xi-y\not\in \text{supp}(K_j)$ for every $y\in {\mathbb R}^n\setminus 3Q$, and hence
$$K_j*(f\chi_{{\mathbb R}^n\setminus 3Q})(\xi)=0.$$
Assume that $\ell_Q\le 2^{j-1}$. Suppose also that $y\in {\mathbb R}^n\setminus 3Q$ and $|y-\xi|\le 2^{j+1}$. Then
$$|y-x|\le |y-\xi|+|\xi-x|\le 2^{j+1}+\sqrt n\ell_Q\le 2^j(2+\sqrt{n}/2),$$
and hence,
\begin{eqnarray*}
|K_j(\xi-y)|&\le& \frac{\|\Omega\|_{L^{\infty}}}{2^{(j-1)n}}\chi_{\{|y-\xi|\le 2^{j+1}\}}(y)\\
&\le& \frac{\|\Omega\|_{L^{\infty}}}{2^{(j-1)n}}\chi_{\{|y-x|\le 2^j(2+\sqrt{n}/2)\}}(y).
\end{eqnarray*}
Therefore,
$$|K_j*(f\chi_{{\mathbb R}^n\setminus 3Q})(\xi)|\le \frac{\|\Omega\|_{L^{\infty}}}{2^{(j-1)n}}\int_{\{|y-x|\le 2^j(2+\sqrt{n}/2)\}}|f(y)|dy$$
Thus,
$$
M_{T_j}(f)(x)\le \frac{\|\Omega\|_{L^{\infty}}}{2^{(j-1)n}}\int_{\{|y-x|\le 2^j(2+\sqrt{n}/2)\}}|f(y)|dy,
$$
which implies (\ref{estmt}).
\end{proof}

Applying Lemma \ref{mtj} yields
\begin{eqnarray}
\sum_{s=1}^{m-1}\sum_{j\in {\mathbb Z}}\|M_{T_j}(B_{j-s})\|_{L^1}&\le& C_n\|\Omega\|_{L^{\infty}}\sum_{s=1}^{m-1}\sum_{j\in {\mathbb Z}}\|B_{j-s}\|_{L^1}\label{sumpart}\\
&\le& C_nm\|\Omega\|_{L^{\infty}}\|f\|_{L^1}.\nonumber
\end{eqnarray}

\subsection{Estimate of ${\mathcal M}_{\la}b$}
Write the set $\{x:{\mathcal M}_{\la}b(x)>\a/4\}$ as the union of the maximal
pairwise disjoint cubes $Q_i\in {\mathcal F}$ with the property
$$
\Big(\sum_{s\ge m}\sum_{j\in {\mathbb Z}}K_j*(B_{j-s}\chi_{{\mathbb R}^n\setminus 3Q_i})\chi_{Q_i}\Big)^*(\la|Q_i|/4)>\a/4,
$$
or, equivalently,
$$
|Q_i|<\frac{4}{\la}|\{x\in Q_i:|\sum_{s\ge m}\sum_{j\in {\mathbb Z}}K_j*(B_{j-s}\chi_{{\mathbb R}^n\setminus 3Q_i})(x)|>\a/4\}|.
$$
It follows that the cubes $Q_i$ can be selected into two disjoint families: let ${\mathcal A}_1$ be the family of $Q_i$ for which
$$
|Q_i|<\frac{8}{\la}|\{x\in Q_i:|\sum_{s\ge m}\sum_{j\in {\mathbb Z}}K_j*B_{j-s}(x)|>\a/8\}|,
$$
and let ${\mathcal A}_2$ be the family of $Q_i$ for which
$$
|Q_i|<\frac{8}{\la}|\{x\in Q_i:|\sum_{s\ge m}\sum_{j\in {\mathbb Z}}K_j*(B_{j-s}\chi_{3Q_i})(x)|>\a/8\}|.
$$

\subsection{The cubes from the first family} We have
\begin{equation}\label{firstf}
\sum_{Q_i\in {\mathcal A}_1}|Q_i|<\frac{8}{\la}
|\{x\in {\mathbb R}^n:|\sum_{s\ge m}\sum_{j\in {\mathbb Z}}K_j*B_{j-s}(x)|>\a/8\}|.
\end{equation}

To estimate the right-hand side here, we use the following result by Seeger \cite{S} (in the next statement we unified Lemmata 2.1 and 2.2 from~\cite{S}).

\begin{lemma}\label{seeger} Let $\{H_j\}$ be a family of functions supported in  $ \{x:2^{j-2}\le |x|\le 2^{j+2}\}$ and such that the estimates
$$\sup_{0\le l\le N}\sup_jr^{n+l}\Big|\Big(\frac{\partial}{\partial r}\Big)^lH_j(r\theta)\Big|\le M_N$$
hold uniformly in $\theta\in S^{n-1}$ and $r>0$. Then for every $0<\kappa<1$ and any natural $s>3$ one can split $H_j=\Gamma_j^s+(H_j-\Gamma_j^s)$
such that the following properties hold.
\begin{enumerate}
\item Let ${\mathcal Q}$ be a collection of pairwise disjoint dyadic cubes, and let ${\mathcal Q}_m=\{Q\in {\mathcal Q}:|Q|=2^{nm}\}, m\in {\mathbb Z}.$ For each $Q\in {\mathcal Q}$ let $f_Q$ be an integrable function supported in $Q$ satisfying $\int|f_Q|dx\le \a|Q|$. Let $F_m=\sum_{Q\in {\mathcal Q}_m}f_Q$.
Then for $s>3$,
$$\Big\|\sum_j\Gamma_j^s*F_{j-s}\Big\|_{L^2}^2\le C_nM_0^22^{-s(1-\kappa)}\a\sum_Q\|f_Q\|_{L^1}.$$
\item Let $Q$ be a cube of sidelength $2^{j-s}$ and let $b_Q$ be integrable and supported in $Q$ with
$\int_Qb_Q=0$. Then for $N\ge n+1$ and $0\le \e\le 1$,
$$\|(H_j-\Gamma_j^s)*b_Q\|_{L^1}\le C_{n,N}(M_02^{-s\e}+M_N2^{s(n+(\e-\kappa)N)})\|b_Q\|_{L^1}.$$
\end{enumerate}
\end{lemma}

Notice that $K_j$ is supported in $\{2^{j-1}\le|x|\le 2^{j+1}\}$ and
$$
\sup_{0\le l\le N}\sup_jr^{n+l}\Big|\Big(\frac{\partial}{\partial r}\Big)^lK_j(r\theta)\Big|\le C_{N,n}\|\Omega\|_{L^{\infty}}.
$$
Therefore, we are in position to apply Lemma~\ref{seeger}. Choose in this lemma $\kappa=\frac{1}{2}$ and $\e=\frac{1}{4}$.
We obtain
\begin{eqnarray}
&&\Big|\Big\{x:|\sum_{s\ge m}\sum_{j\in {\mathbb Z}}K_j*B_{j-s}|>\a/8\Big\}\Big|
\label{step}\\
&&\le
|\{x:|I(x)|>\a/16\}|+|\{x:|II(x)|>\a/16\}|,\nonumber
\end{eqnarray}
where
$$I(x)=\sum_{s\ge m}\sum_{j\in {\mathbb Z}}\Gamma_j^s*B_{j-s}(x)$$
and
$$II(x)=\sum_{s\ge m}\sum_{j\in {\mathbb Z}}(K_j-\Gamma_j^s)*B_{j-s}(x).$$

Observe that
$$\int_P|b_P|dx\le 2\int_P|f|\le 2^{n+1}A\a|P|.$$
Therefore, the first part of Lemma \ref{seeger} yields
\begin{eqnarray*}
&&|\{x:|I(x)|>\a/16\}|\le \frac{256}{\a^2}\|I\|_{L^2}^2\le \frac{256}{\a^2}\Big(
\sum_{s\ge m}\|\sum_j\Gamma_j^s*B_{j-s}\|_{L^2}\Big)^2\\
&&\le \frac{C_n\|\O\|_{L^{\infty}}^2}{\a^2}\Big(\sum_{s\ge m}2^{-s/4}
\Big(A\a\sum_P\|b_P\|_{L^1}\Big)^{1/2}\Big)^2\\
&&\le C_n\|\O\|_{L^{\infty}}^2\frac{A2^{-m/2}}{\a}\|f\|_{L^1}.
\end{eqnarray*}

Applying the second part of Lemma \ref{seeger} with $N=8n$ yields,
\begin{eqnarray*}
&&|\{x:|II(x)|>\a/16\}|\le\frac{16}{\a}\sum_{s\ge m}\|\sum_{j\in {\mathbb Z}}
(K_j-\Gamma_j^s)*B_{j-s}\|_{L^1}\\
&&\le \frac{C_{n}\|\O\|_{L^{\infty}}}{\a}\sum_{s\ge m}
\Big(2^{-s/4}+2^{-ns}\Big)\|f\|_{L^1}\le C_{n}\|\O\|_{L^{\infty}}2^{-m/4}\frac{\|f\|_{L^1}}{\a}.
\end{eqnarray*}

Combining the estimates for $I$ and $II$ with (\ref{firstf}) and (\ref{step}), we obtain
\begin{equation}\label{firstpart}
\sum_{Q_i\in {\mathcal A}_1}|Q_i|\le\frac{C_n}{\la}\Big(\|\O\|_{L^{\infty}}^2A2^{-m/2}+\|\O\|_{L^{\infty}}2^{-m/4}\Big)\frac{\|f\|_{L^1}}{\a}.
\end{equation}

\subsection{The cubes from the second family} Let $Q_i\in {\mathcal A}_2$.
Observe that the cube $3Q_i$ and the cubes appearing in the definition of $B_{j-s}$ are from the same dyadic lattice ${\mathscr D}'$. Therefore,
setting
$$B^{(i)}_{j-s}=\sum_{P:|P|=2^{(j-s)n},P\subset 3Q_i}b_P(x),$$
we obtain that for $2^{j-s}<3\ell_{Q_i}$,
$$B_{j-s}\chi_{3Q_i}=B^{(i)}_{j-s}.$$

Assume that $2^{j-s}\ge 3\ell_{Q_i}$. Then for all $x\in Q_i$ and $y\in 3Q_i$,
$$|x-y|\le 2\sqrt n\ell_{Q_i}\le \frac{4\sqrt n}{3}\frac{1}{2^s}2^{j-1},$$
and therefore, $x-y\not\in\text{supp}\,K_j$, provided $2^s>4\sqrt n/3$.
Hence, assuming that $m$ is such that $2^{m}>4\sqrt n/3$, for all $s\ge m$ we obtain
$$K_j*(B_{j-s}\chi_{3Q_i})(x)=0,$$
which implies
\begin{eqnarray*}
|Q_i|&<&\frac{8}{\la}|\{x\in Q_i:|\sum_{s\ge m}\sum_{j\in {\mathbb Z}}K_j*(B_{j-s}\chi_{3Q_i})(x)|>\a/8\}|\\
&\le& \frac{8}{\la}|\{x:|\sum_{s\ge m}\sum_{j\in {\mathbb Z}}K_j*B^{(i)}_{j-s}(x)|>\a/8\}|.
\end{eqnarray*}

Denote by $\ga$ the constant appearing on the right-hand side of (\ref{firstpart}), that is, let
$$\ga=\frac{C_n}{\la\a}\Big(\|\O\|_{L^{\infty}}^2A2^{-m/2}+\|\O\|_{L^{\infty}}2^{-m/4}\Big).$$
Then, arguing exactly as in the proof of (\ref{firstpart}) and using that all the cubes in the definition of $B^{(i)}_{j-s}$
are supported in $3Q_i$, we obtain
$$|Q_i|\le\ga\int_{3Q_i}|f|.$$
Hence,
$$
\cup_{Q_i\in {\mathcal A}_2}Q_i\subset \{x: Mf(x)>1/(3^n\ga)\},
$$
which implies that the cubes from the second family satisfy the same estimate as (\ref{firstpart}). Therefore,
$$|\{x:{\mathcal M}_{\la}b(x)>\a/4\}|\le
\frac{C_n}{\la}\Big(\|\O\|_{L^{\infty}}^2A2^{-m/2}+\|\O\|_{L^{\infty}}2^{-m/4}\Big)\frac{\|f\|_{L^1}}{\a}.$$

\subsection{Conclusion of the proof} Assume that $2^{m}>4\sqrt n/3$.
Combining the last estimate with (\ref{embed}), (\ref{goodpart}),
(\ref{estem}), (\ref{estol}) and (\ref{sumpart}) yields
\begin{eqnarray*}
&&\|M_{\la,T_{\O}}\|_{L^1\to L^{1,\infty}}\le C_n\Big(\|\O\|_{L^{\infty}}m\\
&&+\frac{1}{\la}\Big(A\|T_{\O}\|_{L^2\to L^2}^2+\frac{9^{n+1}}{A}+\|\O\|_{L^{\infty}}^2A2^{-m/2}+\|\O\|_{L^{\infty}}2^{-m/4}\Big)\Big).
\end{eqnarray*}
From this, setting $\nu=\max(\|\O\|_{L^{\infty}},\|T_{\O}\|_{L^2\to L^2}/\d)$, where $0<\d\le 1$,
and $A=\frac{2^{m/4}}{\nu}$, we obtain
\begin{equation}\label{lest}
\|M_{\la,T_{\O}}\|_{L^1\to L^{1,\infty}}\le C_n\nu\Big(\frac{1}{\la}\big(2^{m/4}\d^2+2^{-m/4}\big)+m\Big).
\end{equation}

One can assume that $\d^{-4}>4\sqrt n/3$ since otherwise Lemma \ref{keying} is trivial.
Then, set in (\ref{lest}) $m\in {\mathbb N}$ such that $2^{m-1}\le \d^{-4}<2^m$. We obtain
$$\|M_{\la,T_{\O}}\|_{L^1\to L^{1,\infty}}\le C_n\nu\Big(\frac{\d}{\la}+\log\frac{2}{\d}\Big),$$
which completes the proof.

\section{A sparse domination principle}
We start with the following general result which can be described in terms of the bi-sublinear maximal operator
${\mathscr M}_T$ defined for a given operator $T$ by
$${\mathscr M}_T(f,g)(x)=\sup_{Q\ni x}\frac{1}{|Q|}\int_Q|T(f\chi_{{\mathbb R}^n\setminus 3Q})||g|dy,$$
where the supremum is taken over all cubes $Q\subset {\mathbb R}^n$ containing $x$.

\begin{theorem}\label{pointsparse}
Let $1\le q\le r$ and $s\ge 1$. Assume that $T$ is a sublinear operator of weak type $(q,q)$, and ${\mathscr M}_T$ maps
$L^r\times L^s$ into $L^{\nu,\infty}$, where $\frac{1}{\nu}=\frac{1}{r}+\frac{1}{s}$.
Then, for every compactly supported $f\in L^r({\mathbb R}^n)$ and every $g\in L^s_{loc}$, there exists a $\frac{1}{2\cdot 3^n}$-sparse family ${\mathcal S}$ such that
\begin{equation}\label{mainin}
|\langle Tf,g \rangle|\le K\sum_{Q\in {\mathcal S}}\langle f\rangle_{r,Q}\langle g\rangle_{s,Q}|Q|,
\end{equation}
where
$$K=C_n\big(\|T\|_{L^q\to L^{q,\infty}}+\|{\mathscr M}_T\|_{L^r\times L^s\to L^{\nu,\infty}}\big).$$
\end{theorem}

\begin{proof} The proof is very similar to the one of \cite[Th. 4.2]{L2}.

Fix a cube $Q_0$. Define a local analogue of ${\mathscr M}_T$ by
$${\mathscr M}_{T,Q_0}(f,g)(x)=\sup_{Q\ni x,Q\subset Q_0}\frac{1}{|Q|}\int_Q|T(f\chi_{3Q_0\setminus 3Q})||g|dy.$$

Consider the sets
$$E_1=\{x\in Q_0: |T(f\chi_{3Q_0})(x)|>A\langle f\rangle_{q,3Q_0}\}$$
and
$$E_2=\{x\in Q_0: {\mathscr M}_{T,Q_0}(f,g)(x)>B\langle f\rangle_{r,3Q_0}\langle g\rangle_{s,Q_0}\},$$
where $A$ and $B$ are chosen in such a way that
$$\max(|E_1|,|E_2|)\le \frac{1}{2^{n+3}}|Q_0|,$$
namely, we take
$$A=(8\cdot 6^n)^{1/q}\|T\|_{L^q\to L^{q,\infty}}\quad\text{and}\quad B=(2^{n+3})^{1/\nu}3^{n/r}\|{\mathscr M}_T\|_{L^r\times L^s\to L^{\nu,\infty}}.$$
Then, the set $\Omega=E_1\cup E_2$ satisfies $|\Omega|\le\frac{1}{2^{n+2}}|Q_0|$.

The Calder\'on-Zygmund decomposition applied to the function $\chi_{\Omega}$ on $Q_0$ at height $\la=\frac{1}{2^{n+1}}$
produces pairwise disjoint cubes $P_j\in {\mathcal D}(Q_0)$ such that
$$\frac{1}{2^{n+1}}|P_j|\le |P_j\cap E|\le \frac{1}{2}|P_j|$$
and $|\Omega\setminus \cup_jP_j|=0$. It follows that $\sum_j|P_j|\le \frac{1}{2}|Q_0|$ and $P_j\cap \Omega^{c}\not=\emptyset$.

Since $|\Omega\setminus \cup_jP_j|=0$, we have
$$\int_{Q_0\setminus \cup_jP_j}|T(f\chi_{3Q_0})||g|\le A\langle f\rangle_{q,3Q_0}\int_{Q_0}|g|.$$
On the other hand, since $P_j\cap \Omega^{c}\not=\emptyset$, we obtain
$$\int_{P_j}|T(f\chi_{3Q_0\setminus 3P_j})||g|\le B\langle f\rangle_{r,3Q_0}\langle g\rangle_{s,Q_0}|P_j|.$$
Combining these estimates along with H\"older's inequality (here we use that $q\le r$ and $s\ge 1$) yields
\begin{eqnarray*}
&&\int_{Q_0}|T(f\chi_{3Q_0})||g|\le \int_{Q_0\setminus \cup_jP_j}|T(f\chi_{3Q_0})||g|\\
&&+\sum_j\int_{P_j}|T(f\chi_{3Q_0\setminus 3P_j})||g|+\sum_j\int_{P_j}|T(f\chi_{3P_j})||g|\\
&&\le(A+B)\langle f\rangle_{r,3Q_0}\langle g\rangle_{s,Q_0}|Q_0|+\sum_j\int_{P_j}|T(f\chi_{3P_j})||g|.
\end{eqnarray*}

Since $\sum_j|P_j|\le \frac{1}{2}|Q_0|$, iterating the above estimate, we obtain that there is a $\frac{1}{2}$-sparse family ${\mathcal F}\subset {\mathcal D}(Q_0)$
such that
\begin{equation}\label{itpart}
\int_{Q_0}|T(f\chi_{3Q_0})||g|\le \sum_{Q\in {\mathcal F}}(A+B)\langle f\rangle_{r,3Q}\langle g\rangle_{s,Q}|Q|
\end{equation}
(notice that ${\mathcal F}=\{P_j^k\},k\in {\mathbb Z}_+$, where
$\{P_j^0\}=\{Q_0\}$, $\{P_j^1\}=\{P_j\}$ and $\{P_j^k\}$ are the cubes obtained at the $k$-th stage of the iterative process).

Take now a partition of ${\mathbb R}^n$ by cubes $R_j$ such that $\text{supp}\,(f)\subset 3R_j$ for each $j$. For example, take a cube $Q_0$ such that
$\text{supp}\,(f)\subset Q_0$ and cover $3Q_0\setminus Q_0$ by $3^n-1$ congruent cubes $R_j$. Each of them satisfies $Q_0\subset 3R_j$. Next, in the same way cover
$9Q_0\setminus 3Q_0$, and so on. The union of resulting cubes, including $Q_0$, will satisfy the desired property.

Having such a partition, apply (\ref{itpart}) to each $R_j$. We obtain a $\frac{1}{2}$-sparse family ${\mathcal F}_j\subset {\mathcal D}(R_j)$ such that
$$
\int_{R_j}|T(f)||g|\le \sum_{Q\in {\mathcal F_j}}(A+B)\langle f\rangle_{r,3Q}\langle g\rangle_{s,Q}|Q|
$$
Therefore,
$$
\int_{{\mathbb R}^n}|T(f)||g|\le \sum_{Q\in {\mathcal \cup_{j}{\mathcal F}_j}}(A+B)\langle f\rangle_{r,3Q}\langle g\rangle_{s,Q}|Q|
$$
Notice that the family $\cup_j{\mathcal F}_j$ is $\frac{1}{2}$-sparse as a disjoint union of $\frac{1}{2}$-sparse families.
Hence, setting ${\mathcal S}=\{3Q: Q\in \cup_j{\mathcal F}_j\}$, we obtain that ${\mathcal S}$ is $\frac{1}{2\cdot 3^n}$-sparse, and
(\ref{mainin}) holds.
\end{proof}

Given $1\le p\le \infty$, define the maximal operator ${\mathscr M}_{p,T}$ by
$${\mathscr M}_{p,T}f(x)=\sup_{Q\ni x}\left(\frac{1}{|Q|}\int_Q|T(f\chi_{{\mathbb R}^n\setminus 3Q})|^pdy\right)^{1/p}$$
(in the case $p=\infty$ set ${\mathscr M}_{p,T}f(x)=M_Tf(x)$).

\begin{cor}\label{mpt}
Let $1\le q\le r$ and $s\ge 1$. Assume that $T$ is a sublinear operator of weak type $(q,q)$, and ${\mathscr M}_{s',T}$ is of weak type $(r,r)$.
Then, for every compactly supported $f\in L^r({\mathbb R}^n)$ and every $g\in L^s_{loc}$, there exists a $\frac{1}{2\cdot 3^n}$-sparse family ${\mathcal S}$ such that
$$
|\langle Tf,g \rangle|\le K\sum_{Q\in {\mathcal S}}\langle f\rangle_{r,Q}\langle g\rangle_{s,Q}|Q|,
$$
where
$$K=C_n\big(\|T\|_{L^q\to L^{q,\infty}}+\|{\mathscr M}_{s',T}\|_{L^r\to L^{r,\infty}}\big).$$
\end{cor}

\begin{proof}
By H\"older's inequality,
$${\mathscr M}_T(f,g)(x)\le {\mathscr M}_{s',T}f(x)M_{s}g(x),$$
where $M_sg=M(g^s)^{1/s}$.
From this, by H\"older's inequality for weak spaces (see \cite[p. 15]{G}) along with the weak type $(s,s)$ estimate for $M_s$,
$$
\|{\mathscr M}_T\|_{L^r\times L^s\to L^{\nu,\infty}}\le C_n\|{\mathscr M}_{s',T}\|_{L^r\to L^{r,\infty}}\quad(1/\nu=1/r+1/s),
$$
which, by Theorem \ref{pointsparse}, completes the proof.
\end{proof}

In order to apply Corollary \ref{mpt} to Theorem \ref{logest}, we first establish a relation between the $L^1\to L^{1,\infty}$ norms of the operators $M_{\la, T}$ and ${\mathscr M}_{p,T}$.

\begin{lemma}\label{equiv} Let $0<\ga\le 1$ and let $T$ be a sublinear operator. The following statements are equivalent:
\begin{enumerate}
\renewcommand{\labelenumi}{(\roman{enumi})}
\item
there exists $C>0$ such that for all $p\ge 1$,
$$\|{\mathscr M}_{p,T}f\|_{L^1\to L^{1,\infty}}\le Cp^{\ga};$$
\item
there exists $C>0$ such that for all $0<\la<1$,
$$\|M_{\la,T}f\|_{L^1\to L^{1,\infty}}\le C\Big(1+\log\frac{1}{\la}\Big)^{\ga}.$$
\end{enumerate}
\end{lemma}

\begin{proof}
Let us show that (i)$\Rightarrow$(ii). By Chebyshev's inequality,
$$M_{\la, T}f(x)\le \frac{1}{\la^{1/p}}{\mathscr M}_{p,T}f(x),$$
which implies
$$\|M_{\la,T}f\|_{L^{1,\infty}}\le \frac{1}{\la^{1/p}}\|{\mathscr M}_{p,T}f\|_{L^{1,\infty}}\le C\frac{p^{\ga}}{\la^{1/p}}\|f\|_{L^1}.$$
Setting here $p=\max(1,\log\frac{1}{\la})$, we obtain (ii).

Turn to the implication (ii)$\Rightarrow$(i). First, observe that
$$\left(\frac{1}{|Q|}\int_Q|T(f\chi_{{\mathbb R}^n\setminus 3Q})|^pdy\right)^{1/p}=\left(\int_0^1\big(T(f\chi_{{\mathbb R}^n\setminus 3Q})\chi_Q\big)^*\big(\la|Q|\big)^pd\la\right)^{1/p},$$
which implies
\begin{equation}\label{po}
{\mathscr M}_{p,T}f(x)\le \left(\int_0^1M_{\la,T}f(x)^pd\la\right)^{1/p}.
\end{equation}

For $N>0$ denote
$$G_{p,T,N}f(x)=\left(\int_0^1\min(M_{\la,T}f(x),N)^pd\la\right)^{1/p}.$$
Set also $\mu_{f}(\a,R)=|\{|x|\le R:|f(x)|>\a\}|$ for $R,\a>0$.

Let $k>1$ that will be specified later. By H\"older's inequality,
\begin{eqnarray*}
G_{p,T,N}f(x)&\le& \left(\int_0^{1/2^{kp}}\min(M_{\la,T}f(x),N)^pd\la\right)^{1/p}
+M_{1/2^{kp},T}f(x)\\
&\le& \frac{1}{2^{k-1}}G_{kp,T,N}f(x)+M_{1/2^{kp},T}f(x).
\end{eqnarray*}
Hence,
\begin{eqnarray*}
\mu_{G_{p,T,N}f}(\a, R)&\le& \mu_{G_{kp,T,N}f}(2^{k-2}\a,R)+\mu_{M_{1/2^{kp},T}f}(\a/2,R)\\
&\le& \mu_{G_{kp,T,N}f}(2^{k-2}\a,R)+C\frac{(kp)^{\ga}}{\a}\|f\|_{L^1}.
\end{eqnarray*}

Iterating this estimate, after the $j$-th step we obtain
$$
\mu_{G_{p,T,N}f}(\a,R)\le \mu_{G_{k^jp,T,N}f}(2^{(k-2)j}\a,R)+\frac{2^{k-2}C}{\a}\sum_{i=1}^j\Big(\frac{k^{\ga}}{2^{k-2}}\Big)^ip^{\ga}\|f\|_{L^1}.
$$
Take here $k=5$. Since $G_{p,T,N}$ is bounded uniformly in $p$, we obtain that
$\mu_{G_{5^jp,T,N}f}(8^j\a,R)=0$ starting from some $j$ big enough.
Hence, letting $j\to\infty$ in the above estimate yields
$$
\mu_{G_{p,T,N}f}(\a,R)\le \frac{C}{\a}\sum_{i=1}^{\infty}\Big(\frac{5}{8}\Big)^ip^{\ga}\|f\|_{L^1}\le \frac{C'}{\a}p^{\ga}\|f\|_{L^1}.
$$
Letting here $N,R\to \infty$ and applying (\ref{po}) completes the proof.
\end{proof}

Now, we are ready to show that Theorem \ref{logest} provides a different approach to (\ref{rough}).

\begin{cor}\label{rou} Let $T_{\O}$ be a rough homogeneous singular integral with $\O\in L^{\infty}(S^{n-1})$.
Then, for every compactly supported $f\in L^p({\mathbb R}^n)$ and every $g\in L^1_{loc}$, there exists a $\frac{1}{2\cdot 3^n}$-sparse family ${\mathcal S}$ such that
$$
|\langle T_{\O}f,g \rangle|\le C_np'\|\O\|_{L^{\infty}(S^{n-1})}\sum_{Q\in {\mathcal S}}\langle f\rangle_{p,Q}\langle g\rangle_{1,Q}|Q|\quad(p>1).
$$
\end{cor}

\begin{proof} By Theorem \ref{logest} along with Lemma \ref{equiv} with $\ga=1$,
$$\|{\mathscr M}_{p,T_{\O}}\|_{L^1\to L^{1,\infty}}\le C_n\|\O\|_{L^{\infty}(S^{n-1})}p\quad(p\ge 1).$$
Also, by \cite{S}, $\|T_{\O}\|_{L^1\to L^{1,\infty}}\le C_n\|\O\|_{L^{\infty}(S^{n-1})}.$
Hence, by Corollary \ref{mpt} with $q=r=1$ and $s=p$, there exists a $\frac{1}{2\cdot 3^n}$-sparse family ${\mathcal S}$ such that
$$
|\langle T_{\O}f,g \rangle|\le C_np'\|\O\|_{L^{\infty}(S^{n-1})}\sum_{Q\in {\mathcal S}}\langle f\rangle_{1,Q}\langle g\rangle_{p,Q}|Q|\quad(p>1).
$$
Since the operator $T_{\O}$ is essentially self-adjoint, the same estimate holds for the adjoint operator $T_{\O}^*$. Replacing in the above estimate
$T_{\O}$ by $T_{\O}^*$ and interchanging $f$ and $g$ completes the proof.
\end{proof}

\begin{remark}
Theorem \ref{pointsparse} and Corollary \ref{mpt} can be easily generalized by means of replacing the normalized $L^p$ averages by
the normalized Orlicz averages $\|f\|_{\f,Q}$ defined by
$$
\|f\|_{\f,Q}=\inf\Big\{\a>0:\frac{1}{|Q|}\int_Q\f(|f(y)|/\a)dy\le 1\Big\}.
$$

We mention only one interesting particular case of such a generalization.
Denote $\|f\|_{L\log L,Q}$ if $\f(t)=t\log({\rm{e}}+t)$ and $\|f\|_{\text{exp}L,Q}$ if $\f(t)=e^t-1$.
Given an operator $T$ define the maximal operator $M_{\text{exp} L,T}$ by
$$M_{\text{exp} L,T}f(x)=\sup_{Q\ni x}\|T(f\chi_{{\mathbb R}^n\setminus 3Q})\|_{\text{exp} L,Q}.$$
Then if $T$ and $M_{\text{exp} L,T}$ are of weak type $(1,1)$, for every appropriate $f$ and $g$, there exists a sparse family ${\mathcal S}$ such
that
$$
|\langle Tf,g \rangle|\le K\sum_{Q\in {\mathcal S}}\langle f\rangle_{1,Q}\|g\|_{L\log L,Q}|Q|,
$$
where $K=C_n(\|T\|_{L^1\to L^{1,\infty}}+\|M_{\text{exp} L,T}\|_{L^1\to L^{1,\infty}})$.

In particular, we conjecture that if $T_{\O}$ is a rough homogeneous singular integral with $\O\in L^{\infty}(S^{n-1})$, then $M_{\text{exp} L,T_{\O}}$
is of weak type $(1,1)$. This would imply a small improvement of (\ref{rough}) with $p'\langle f\rangle_{p,Q},p>1,$ replaced by $\|f\|_{L\log L,Q}$.
\end{remark}

\end{document}